\documentclass[11pt]{amsart}
\usepackage{amsmath, amssymb, amscd, mathrsfs, url, pinlabel,hyperref, verbatim}
\usepackage[margin=1.25in,centering,letterpaper,dvips]{geometry}
\usepackage{color,dcpic,latexsym,graphicx,epstopdf,comment}
\usepackage[all]{xy}
\usepackage{tikz}
\usepackage{enumitem,upquote}

\title{Khovanov homology detects the Hopf links}

\author{John A. Baldwin}
\email{john.baldwin@bc.edu}
\address{Department of Mathematics\\Boston College}

\author{Steven Sivek}
\email{s.sivek@imperial.ac.uk}
\address{Department of Mathematics\\Imperial College London}

\author{Yi Xie}
\email{yxie@scgp.stonybrook.edu}
\address{Simons Center for Geometry and Physics\\State University of New York}

\thanks{JAB was supported by NSF CAREER Grant DMS-1454865.}

\newtheorem {theorem}{Theorem}

\newtheorem {proposition}[theorem]{Proposition}

\numberwithin{equation}{section}
\numberwithin{theorem}{section}

\theoremstyle{definition}

\newlist{pcases}{enumerate}{1}
\setlist[pcases]{
  label=\bf{Case~\arabic*:}\protect\thiscase.~,
  ref=\arabic*,
  align=left,
  labelsep=0pt,
  leftmargin=0pt,
  labelwidth=0pt,
  parsep=0pt
}
\newcommand{\case}[1][]{%
  \if\relax\detokenize{#1}\relax
    \def\thiscase{}%
  \else
    \def\thiscase{~#1}%
  \fi
  \item
}

\newcommand{\Z}{\mathbb{Z}}

\newcommand{\C}{\mathbb{C}}
\newcommand{\F}{\mathbb{F}}
\newcommand{\Q}{\mathbb{Q}}

\newcommand{\rank}{\operatorname{rank}}

\newcommand{\bK}{\mathbb{K}}

\newcommand\SHI{\mathit{SHI}}

\newcommand\KHI{\mathit{KHI}}

\newcommand\Kh{\mathit{Kh}}
\newcommand\Khr{\Kh_{\mathrm{red}}}

\newcommand\mirror[1]{\overline{#1}}

\DeclareFontFamily{U}{mathx}{\hyphenchar\font45}
\DeclareFontShape{U}{mathx}{m}{n}{
      <5> <6> <7> <8> <9> <10>
      <10.95> <12> <14.4> <17.28> <20.74> <24.88>
      mathx10
      }{}
\DeclareSymbolFont{mathx}{U}{mathx}{m}{n}
\DeclareFontSubstitution{U}{mathx}{m}{n}
\DeclareMathAccent{\widecheck}{0}{mathx}{"71}

\newcommand{\lk}{\operatorname{lk}}

\tikzset{every picture/.style=thick}

\begin{document}

\begin{abstract}
We prove that any link in $S^3$ whose Khovanov homology is the same as that of a Hopf link must be isotopic to that Hopf link.  This holds for both reduced and unreduced Khovanov homology, and with coefficients in either $\Z$ or $\Z/2\Z$.
\end{abstract}

\maketitle

Khovanov homology \cite{khovanov} associates to each link $L \subset S^3$ a bigraded group $\Kh^{*,*}(L)$, whose graded Euler characteristic recovers the Jones polynomial $V_L(q)$, as well as a reduced variant $\Khr^{*,*}(L)$ \cite{khovanov-patterns}.  It is known to detect the unknot \cite{km-unknot}, the $n$-component unlink for all $n$ \cite{batson-seed}, and the trefoils \cite{bs-trefoils}.  In this note we prove the same for the Hopf links $H_\pm$, which are oriented so that the two components have linking number $\pm 1$.  Then
\begin{align*}
\Kh(H_+;\Z) &\cong \Z_{(0,0)} \oplus \Z_{(0,2)} \oplus \Z_{(2,4)} \oplus \Z_{(2,6)} \\
\Kh(H_-;\Z) &\cong \Z_{(0,0)} \oplus \Z_{(0,-2)} \oplus \Z_{(-2,-4)} \oplus \Z_{(-2,-6)} 
\end{align*}
where $\Z_{(h,q)}$ is a copy of $\Z$ in bigrading $(h,q)$.  The reduced Khovanov homology of a link depends on a choice of distinguished component, which we generally suppress from the notation, but we have $\Khr(H_+;\Z) \cong \Z_{(0,1)} \oplus \Z_{(2,5)}$ and $\Khr(H_-;\Z) \cong \Z_{(0,-1)} \oplus \Z_{(-2,-5)}$ regardless of this choice.

\begin{theorem} \label{thm:main}
Let $L$ be a link in $S^3$ such that either $\Kh(L) \cong \Kh(H_\pm)$ or $\Khr(L) \cong \Khr(H_\pm)$ as bigraded groups, with coefficients in either $\Z$ or $\Z/2\Z$.  Then $L = H_\pm$.
\end{theorem}

Our proof makes use of a handful of spectral sequences involving Khovanov homology.  Batson and Seed's link splitting spectral sequence \cite{batson-seed}, together with Kronheimer and Mrowka's spectral sequence converging to singular instanton knot homology \cite{km-unknot}, will tell us that such a link $L$ must have exactly two unknotted components with linking number $\pm 1$.  Once we know the module structure on $\Khr(L)$, we apply a refinement by the third author \cite{xie-point} of the latter spectral sequence to determine that the instanton knot homology $\KHI(L)$ of \cite{km-excision} has rank at most 4.  We will then determine the Alexander grading on $\KHI(L)$ to see that $L$ has Seifert genus zero, and conclude that $L$ must be a Hopf link.

If $L$ is a link with $r$ components, then $\Kh^{*,*}(L)$ is invariant as a module over the ring
\[ R_r = \Z[x_1,\dots,x_r]/\langle x_1^2,\dots,x_r^2\rangle \]
\cite{khovanov-patterns,hedden-ni}, in which each $x_i$ preserves the $h$ (``homological'') grading while lowering the $q$ (``quantum'') grading.  If we define the reduced Khovanov homology of $L$ using the component which corresponds to $x_r$, then $\Khr^{*,*}(L)$ is a module over $R_{r-1}$.  We begin by determining this module structure.

\begin{proposition} \label{prop:khr-q}
Let $L$ be a link in $S^3$, and let $H$ be either $H_+$ or $H_-$.  Let $\F = \Z/2\Z$. Suppose that any one of the following is true as an isomorphism of bigraded groups:
\begin{enumerate}
\item $\Kh(L;\Z) \cong \Kh(H;\Z)$. \label{i:unr-z}
\item $\Kh(L;\F) \cong \Kh(H;\F)$. \label{i:unr-f}
\item $\Khr(L;\Z) \cong \Khr(H;\Z)$ for some choice of component of $L$. \label{i:red-z}
\item $\Khr(L;\F) \cong \Khr(H;\F)$ for some choice of component of $L$. \label{i:red-f}
\end{enumerate}
Then $L$ has exactly two components, which are both unknots, and if $H=H_\pm$ then their linking number is $\pm1$.  Moreover, we have $\rank \Khr(L;\Z) = 2$, with
\[ \Khr(L;\Q) \cong \begin{cases} \Q_{(0,1)} \oplus \Q_{(2,5)}, & H=H_+ \\ \Q_{(0,-1)} \oplus \Q_{(-2,-5)}, & H=H_-, \end{cases} \]
and the $R_1$-action on $\Khr(L;\Q)$ is trivial.
\end{proposition}

\begin{proof}
Suppose that $H=H_+$; the case of $H_-$ is identical.   We first claim that each of conditions \eqref{i:unr-z}, \eqref{i:unr-f}, and \eqref{i:red-z} implies condition \eqref{i:red-f}, and then use the latter to prove the rest of the proposition.  Certainly \eqref{i:unr-z} implies \eqref{i:unr-f} and \eqref{i:red-z} implies \eqref{i:red-f} by the universal coefficient theorem.  Moreover, \eqref{i:unr-f} and \eqref{i:red-f} are equivalent by the identity
\[ \Kh^{h,q}(K;\F) \cong \Khr^{h,q-1}(K;\F) \oplus \Khr^{h,q+1}(K;\F) \]
where $K$ is an arbitrary link \cite[Corollary~3.2.C]{shumakovitch}, so this proves the claim.

We suppose from now on that condition~\eqref{i:red-f} holds; since this is equivalent to \eqref{i:unr-f}, we have
\begin{align*}
\Kh(L;\F) &\cong \F_{(0,0)} \oplus \F_{(0,2)} \oplus \F_{(2,4)} \oplus \F_{(2,6)} \\
\Khr(L;\F) &\cong \F_{(0,1)} \oplus \F_{(2,5)}.
\end{align*}
If $L$ were a knot, then the rank of $\Khr(L;\F)$ would be congruent mod 2 to $V_L(-1) = \pm \det(L) \equiv 1 \pmod{2}$, so $L$ must be a link, say $L = K_1 \cup K_2 \cup \dots \cup K_r$ with $r \geq 2$.

We apply the rank inequality \cite[Corollary~1.6]{batson-seed} derived from Batson and Seed's link splitting spectral sequence, namely that if $\bK$ is any field then
\begin{equation} \label{eq:bs-inequality}
\rank \Kh(L;\bK) \geq \rank \bigotimes_{i=1}^r \Kh(K_i;\bK).
\end{equation}
For $\bK = \F$, the left side is $4$ while the right side is at least $2^r$ with equality if and only if $K_i$ is unknotted \cite{km-unknot}, so $r=2$ and $K_1$ and $K_2$ are both unknotted.  (See \cite[Proposition~7.1]{batson-seed}.)  In fact, this inequality respects the grading $\ell=h-q$: if we let $t=2\lk(K_1,K_2)$, then \cite[Corollary~4.4]{batson-seed} says that
\[ \rank^\ell \Kh(L;\F) \geq \rank^{\ell+t} \left(\bigotimes_i \Kh(K_i;\F)\right) = \rank^{\ell+t} \big(\F_{(0,-1)}\oplus \F_{(0,1)}\big)^{\otimes 2}. \]
The left side has ranks $1,2,1$ in $\ell$-gradings $0,-2,-4$ respectively, and the right side has rank 2 when $\ell+t=0$, so we must have $-2 + t=0$, or $\lk(K_1,K_2) = \frac{1}{2}t=1$.

Finally, applying the inequality \eqref{eq:bs-inequality} with $\bK = \Q$ says that $\rank \Kh(L;\Q) \geq 4$,
and there is an exact triangle
\[ \dots \to \Kh(L;\Q) \to \Khr(L;\Q) \to \Khr(L;\Q) \to \dots, \]
so we must have $\rank \Khr(L;\Q) \geq 2$.  Applying the universal coefficient theorem repeatedly, we determine first that $\rank \Khr(L;\Z) \geq 2$, and that if this inequality is strict then $\rank \Khr(L;\F) > 2$ as well, which is false.  Thus $\rank \Khr(L;\Z) = 2 = \rank \Khr(L;\F)$, and so the $\Z$ summands of the former have the same bigrading as the $\F$ summands of the latter.  We deduce from this that
\[ \Khr(L;\Q) \cong \Q_{(0,1)} \oplus \Q_{(2,5)}. \]
Now if we pick a distinguished component of $L$ so that $\Khr(L;\Q)$ is a module over $R_1 = \Z[x_1]/\langle x_1^2\rangle$, then $x_1$ acts with square zero on each $\Khr^{h,*}(L;\Q)$, which is either $0$ or $\Q$, and so it must act trivially as claimed.
\end{proof}

We will now make use of Kronheimer and Mrowka's \emph{instanton knot homology} $\KHI$, defined in \cite{km-excision} and extended to links in \cite{km-alexander}.  They showed in \cite{km-unknot} that for any knot $K$, the complex vector space $\KHI(K)$ is isomorphic to the reduced singular instanton knot homology $I^\natural(K;\C)$, and using a spectral sequence $\Kh(\mirror{K}) \Rightarrow I^\natural(K)$ they deduced a rank inequality
\[ \rank \Khr(K) \geq \rank \KHI(K). \]
Based on results in \cite{km-excision}, this proved that $\Khr$ has rank 1 if and only if $K$ is unknotted.  The third author \cite{xie-point} extended this from knots to pointed links and incorporated the module structure on Khovanov homology to prove the following.
\begin{theorem}[{\cite[Theorem~5.4]{xie-point}}] \label{thm:xie-inequality}
Let $L$ be a link of $r$ components in $S^3$, and fix a base point on the $r$th component, equipping $\Khr(L;\Z)$ with an $R_{r-1}$-module structure.  Let $X' = (x_1,\dots,x_{r-1}) \in R_{r-1}^{\oplus r-1}$.  Then
\[ \rank \KHI(L) \leq \rank H^*(\Khr(L;\Z) \otimes_{R_{r-1}} K(X')), \]
where $K(X')$ is the Koszul complex
\[ 0 \to R_{r-1} \xrightarrow{\wedge X'} \Lambda^1 R_{r-1}^{\oplus r-1} \xrightarrow{\wedge X'} \Lambda^2 R_{r-1}^{\oplus r-1} \xrightarrow{\wedge X'} \dots \xrightarrow{\wedge X'} \Lambda^{r-1} R_{r-1}^{\oplus r-1} \to 0. \]
\end{theorem}

\begin{proposition} \label{prop:rank-koszul}
Let $L \subset S^3$ be a link satisfying any of the hypotheses of Proposition~\ref{prop:khr-q}.  Then $\rank \KHI(L) \leq 4$.
\end{proposition}

\begin{proof}
Proposition~\ref{prop:khr-q} tells us that $L$ is a 2-component link, and that $\Khr(L;\Q) \cong \Q^2$ has a trivial action of $R_1 = \Z[x_1]/\langle x_1^2\rangle$.  In this case the Koszul complex $K(X')$ is
\[ 0 \to R_1 \xrightarrow{\cdot x_1} R_1 \to 0, \]
and $\Khr(L;\Q) \cong (R_1/\langle x_1\rangle \otimes_\Z \Q)^{\oplus 2}$ as $R_1$-modules.  The complex
\[ (R_1/\langle x_1\rangle \otimes_\Z \Q) \otimes_{R_1} K(X') = \big( 0 \to \Q \xrightarrow{0} \Q \to 0 \big) \]
has cohomology $\Q^2$, so we conclude that
\[ \rank H^*(\Khr(L;\Z) \otimes_{R_1} K(X')) = 4 \]
and the proposition now follows from Theorem~\ref{thm:xie-inequality}.
\end{proof}

For links in $S^3$ (and more generally for null-homologous links with a choice of Seifert surface), instanton knot homology is equipped with an Alexander grading
\[ \KHI(L) = \bigoplus_{j=-g+1-r}^{g-1+r} \KHI(L,j), \]
where $L$ has $r$ components and Seifert genus $g$.  This grading is symmetric in the sense that $\KHI(L,j) \cong \KHI(L,-j)$ for all $j$.  Each $\KHI(L,j)$ is canonically $\Z/2\Z$-graded, and the Euler characteristics of these summands determine the Alexander polynomial of $L$ by
\begin{equation} \label{eq:euler-characteristic}
\sum_j \chi(\KHI(L,j))t^j = -(t^{1/2} - t^{-1/2})^{r-1} \Delta_L(t).
\end{equation}
See \cite[Theorem~3.6]{km-alexander} or \cite{lim}.

Let $\tilde\Delta_L(t_1,t_2)$ denote the multivariable Alexander polynomial of a 2-component link $L = K_1 \cup K_2$.  This is related to the single-variable Alexander polynomial $\Delta_L(t)$ by
\[ \Delta_L(t) = \pm(t^{1/2}-t^{-1/2}) \tilde\Delta_L(t,t), \]
see e.g.\ \cite[Lemma~10.1]{milnor}.  It follows from \eqref{eq:euler-characteristic} that
\[ \sum_j \chi(\KHI(L,j)) t^j = \pm (t^{1/2}-t^{-1/2})^2 \tilde\Delta_L(t,t), \]
and hence that the right side of
\begin{equation} \label{eq:multi-khi}
\tilde\Delta_L(t,t) = \pm \left(\frac{\sum_j \chi(\KHI(L,j)) t^j}{t-2+t^{-1}}\right)
\end{equation}
must be a Laurent polynomial, i.e.\ that $t-2+t^{-1}$ divides the numerator.  Torres \cite{torres} proved that $\tilde\Delta_L(t_1,t_2)$ also satisfies $\tilde\Delta_L(1,1) = \pm \lk(K_1,K_2)$.

\begin{proposition} \label{prop:khi-alexander}
Let $L = K_1 \cup K_2$ be a two-component link with $\lk(K_1,K_2)=\pm 1$.  If $\rank \KHI(L) \leq 4$, then 
\[ \KHI(L) = \C_1 \oplus (\C^{\oplus 2})_0 \oplus C_{-1}, \]
where the subscripts denote the Alexander grading, and $\Delta_L(t) = \pm (t^{1/2} - t^{-1/2})$.
\end{proposition}

\begin{proof}
We first claim that in fact $\KHI(L)$ is nonzero in at least three different Alexander gradings.  If not, then \eqref{eq:multi-khi} has the form
\[ \tilde\Delta_L(t,t) = \pm \left(\frac{ c_1t^{e_1} + c_2t^{e_2} }{t - 2 + t^{-1}}\right) \]
for some integers $c_1,c_2$ which are not both zero and some integers $e_1 \neq e_2$.  (The numerator cannot be identically zero because then $\tilde\Delta_L(t,t)=0$, contradicting $\tilde\Delta_L(1,1) = \pm 1$.)  The denominator has a double root at $t=1$, so it cannot divide the numerator, which is either a monomial (and is thus nonzero when $t=1$) or a binomial with only simple roots, and we have a contradiction.  

Next, we show that $\rank \KHI(L) = 4$.  It is already at least three, since it has positive rank in each of at least three Alexander gradings.  But the rank must be even, since it has the same parity as $\sum_j \chi(\KHI(L,j))(-1)^j$ and this sum is equal to $\pm 4\tilde\Delta_L(-1,-1) \in 4\Z$ by setting $t=-1$ in \eqref{eq:multi-khi}.  Thus the total rank is 4 as claimed.

We let $m > k$ be the two largest Alexander gradings in which $\KHI(L)$ is nonzero.  By symmetry $-m$ is the lowest such grading, so $-m < k < m$; then $\KHI(L,-k) \neq 0$ as well, so we must have had $k \geq 0$.  If $k > 0$ then $\KHI(L,j)$ has positive rank for each of $j=\pm m,\pm k$, and its total rank is 4, so it must have rank exactly one in each of these gradings.  Otherwise $k=0$, and by symmetry we have $\rank(\KHI(L)) \equiv \rank(\KHI(L,0)) \pmod{2}$; the former is even, so the latter is as well, and $\KHI(K,j)$ now has rank at least $1,2,1$ for each of $j=m,0,-m$ and total rank $4$.  In either case we conclude that
\[ \KHI(L) \cong \C_m \oplus \C_k \oplus \C_{-k} \oplus \C_{-m} \]
where $m > k \geq 0$.

We do not know the $\Z/2\Z$ grading of each summand, but from \eqref{eq:multi-khi} 
there must be signs $\epsilon_m,\epsilon_k,\epsilon_{-k},\epsilon_{-m} \in \{\pm 1\}$ such that
\[ \tilde\Delta_L(t,t) = \pm\left(\frac{\epsilon_m t^m + \epsilon_k t^k + \epsilon_{-k} t^{-k} + \epsilon_{-m} t^{-m}}{t-2+t^{-1}}\right). \]
The numerator must be a multiple of $(t-1)^2$, so both it and its derivative are zero at $t=1$, giving us the conditions
\begin{align*}
\epsilon_m + \epsilon_k + \epsilon_{-k} + \epsilon_{-m} &= 0, &
m(\epsilon_m - \epsilon_{-m}) + k(\epsilon_k - \epsilon_{-k}) &= 0.
\end{align*}
From the second equation we have $\epsilon_m = \epsilon_{-m}$, or else the first term would have magnitude $2m$ and thus be strictly greater than $|k(\epsilon_k-\epsilon_{-k})| \leq 2k$; and from $\sum \epsilon_j=0$ we now have $\epsilon_k + \epsilon_{-k} = -2\epsilon_m$, so that $\epsilon_k = \epsilon_{-k} = -\epsilon_m$.  We conclude that
\[ \tilde\Delta_L(t,t) = \pm \left(\frac{t^m - t^k - t^{-k} + t^{-m}}{t-2+t^{-1}} \right). \]

Finally, we use the fact that $\tilde\Delta_L(1,1) = \pm 1$ to determine $m$ and $k$.  Given an equation involving Laurent polynomials of the form $p(t) = t^c(t-1)^2q(t)$ we have $q(1) = \frac{1}{2}p''(1)$, and so taking $p(t) = \pm(t^m-t^k-t^{-k}+t^{-m})$ and $q(t) = \tilde\Delta_L(t,t)$ gives us
\[ \tilde\Delta_L(1,1) = \pm \frac{1}{2}\left.\frac{d^2}{dt^2}(t^m-t^k-t^{-k}+t^{-m})\right|_{t=1} = \pm(m^2-k^2). \]
Since $m > k \geq 0$, this can only be equal to $\pm 1$ if $m=1$ and $k=0$.  Thus $\KHI(L)$ is exactly as claimed, while $\tilde\Delta_L(t,t) = \pm 1$ and $\Delta_L(t) = \pm(t^{1/2}-t^{-1/2})\tilde\Delta_L(t,t) = \pm(t^{1/2}-t^{-1/2})$.
\end{proof}

We now know enough about any link $L$ with the same Khovanov homology as a Hopf link $H_\pm$ to determine its link type.

\begin{proof}[Proof of Theorem~\ref{thm:main}]
We have $L = K_1 \cup K_2$, with both $K_i$ unknotted and $\lk(K_1,K_2)=\pm 1$, by Proposition~\ref{prop:khr-q}.   Proposition~\ref{prop:rank-koszul} gives us the bound $\rank \KHI(L) \leq 4$, so then
\[ \KHI(L) \cong \C_1 \oplus (\C^{\oplus 2})_0 \oplus \C_{-1} \]
by Proposition~\ref{prop:khi-alexander}, where again the subscripts denote the Alexander grading.

We now claim that the Seifert genus of $L$ is $0$, as a consequence of the more general
\[ g(L)+r-1 = \max \{j \mid \KHI(L,j) \neq 0 \} \]
for $r$-component links in $S^3$ with irreducible complement.  This generalizes the case $r=1$ of \cite[Proposition~7.16]{km-excision}, and the proof is essentially the same.  If $\Sigma$ is a genus-$g$ Seifert surface for $L$, then we identify the appropriate Alexander grading with a sutured instanton homology group,
\[ \KHI(L,g+r-1) \cong \SHI(S^3(\Sigma)), \]
where $S^3(\Sigma)$ is obtained by cutting open the complement of $L$ (with a pair of meridional sutures on each component) along $\Sigma$.  (The reason for the shift by $r-1$ is that the $j$th Alexander grading is the generalized $2j$-eigenspace of an operator $\mu(\bar\Sigma)$ on the instanton homology of some closed manifold, and the maximal real eigenvalue of $\mu(\bar\Sigma)$ is $2g(\bar\Sigma)-2 = 2(g+r-1)$, as in \cite[\S 2.5]{km-alexander}.)  If $\Sigma$ is genus-minimizing then $S^3(\Sigma)$ is a taut sutured manifold, so $\SHI(S^3(\Sigma)) \neq 0$ by \cite[Theorem~7.12]{km-excision} and the claim follows.

Thus $L$ bounds an annulus $A$, and the core of $A$ is isotopic to either of the unknots $K_i$ on its boundary.  The boundary is in particular a cable of the unknot, and the only such 2-component links with linking number $\pm1$ are the positive and negative Hopf links.
\end{proof}

\bibliographystyle{alpha}
\bibliography{References}

\end{document}